\documentclass[11pt,oneside,british,reqno]{amsart}
\usepackage[T1]{fontenc}
\usepackage[latin9]{inputenc}
\usepackage{color}
\usepackage{babel}
\usepackage{amstext}
\usepackage{amsthm}
\usepackage{amssymb}
\usepackage{setspace}
\usepackage[unicode=true,pdfusetitle,
 bookmarks=true,bookmarksnumbered=false,bookmarksopen=false,
 breaklinks=false,pdfborder={0 0 0},pdfborderstyle={},backref=false,colorlinks=true]
 {hyperref}
\usepackage{breakurl}

\makeatletter
  \theoremstyle{plain}
  \newtheorem*{thm*}{\protect\theoremname}
\theoremstyle{plain}
\newtheorem{thm}{\protect\theoremname}
  \theoremstyle{plain}
  \newtheorem{lem}[thm]{\protect\lemmaname}

\usepackage{calrsfs}
\usepackage{graphicx}
\usepackage{color}
\usepackage{perpage}
\usepackage{upquote}
\usepackage{bbm}
\usepackage[shortlabels]{enumitem}

\MakePerPage{footnote}
\gdef\SetFigFontNFSS#1#2#3#4#5{} %Silence pointless warnings due to xfig
\gdef\SetFigFont#1#2#3#4#5{} %Silence pointless warnings due to xfig
\def\clap#1{\hbox to 0pt{\hss#1\hss}}

\DeclareMathOperator{\codim}{codim}

\definecolor{myblue}{rgb}{0.09,0.32,0.44} %22-84-113
\hypersetup{pdfborder={0 0 0},pdfborderstyle={},colorlinks=true,linkcolor=myblue,citecolor=myblue,urlcolor=blue}

\theoremstyle{remark}
\newtheorem*{qst*}{Question}
\newtheorem*{rmrks*}{Remarks}

\newlength{\tempindent} 
\newcommand{\lazyenum}{
\setlength{\tempindent}{\parindent} 
\begin{enumerate}[leftmargin=0cm,itemindent=0.7cm,labelwidth=\itemindent,labelsep=0cm,align=left,label=\arabic*)]
\setlength{\parskip}{\smallskipamount}
\setlength{\parindent}{\tempindent}
}

\makeatletter
\renewcommand{\andify}{%
  \nxandlist{\unskip, }{\unskip{} \@@and~}{\unskip{} \@@and~}}
\def\author@andify{%
  \nxandlist {\unskip ,\penalty-1 \space\ignorespaces}%
    {\unskip {} \@@and~}%
    {\unskip \penalty-2 \space \@@and~}%
}
\let\@wraptoccontribs\wraptoccontribs
\makeatother

\ifpdf
\def\afs#1#2{\href{#1}{\nolinkurl{#2}}}
\else
\def\afs#1#2{\burlalt{#1}{#2}}
\fi
\makeatother

  \providecommand{\lemmaname}{Lemma}
  \providecommand{\theoremname}{Theorem}
\providecommand{\theoremname}{Theorem}

\begin{document}

\title[A random determinant is not a square]{Supplement to the article ``irreducible polynomials with bounded
height''}

\author{Lior Bary-Soroker and Gady Kozma}

\address{Raymond and Beverly Sackler School of Mathematical Sciences, Tel
Aviv University, Tel Aviv 69978, Israel.}

\email{barylior@post.tau.ac.il}

\address{Department of Mathematics, The Weizmann Institute of Science, Rehovot
76100, Israel.}

\email{gady.kozma@weizmann.ac.il}

\maketitle
In this paper we prove that the determinant of a random matrix is
unlikely to be a square. Formally,
\begin{thm*}
\label{thm:deter}Let $M$ be an $n\times n$ matrix with i.i.d.\ entries
taking the value 0 with probability $\frac{1}{2}$ and the values
$1$ and $-1$ with probability $\frac{1}{4}$ each. Then 
\[
\lim_{n\to\infty}\mathbb{P}(\exists k\in\mathbb{Z}\text{ such that }\det M=k^{2})=0.
\]
\end{thm*}
We direct the reader to our main paper \cite[\S 4]{BK} for motivation
for such a result, and in particular why we are interested in squares
and not in any other sparse subset of the integers.

Throughout the paper we denote by $\xi_{i}$ random independent variables
taking the value 0 with probability $\frac{1}{2}$ and the values
$1$ and $-1$ with probability $\frac{1}{4}$ each.
\begin{lem}
\label{lem:Fourier}For any $a_{1},\dotsc,a_{n}\in\mathbb{Z}$ and
any $x\in\mathbb{Z}$,
\[
\mathbb{P}\Big(\sum_{i=1}^{n}\xi_{i}a_{i}=x\Big)\le\mathbb{P}\Big(\sum_{i=1}^{n}\xi_{i}a_{i}=0\Big).
\]
\end{lem}

\begin{proof}
This follows immediately from the fact that the Fourier transform
$\phi$ of the distribution function is positive since then the right-hand
side is $\int\phi(t)\,dt$ while the right-hand side is 
\[
\int\phi(t)e^{-ixt}\,dt\le\int|\phi(t)|\,dt=\int\phi\,dt.
\]
To see that $\phi$ is positive, note that it is a product of Fourier
transforms of the individual summands, and each one is simply $\frac{1}{2}(1+\cos(a_{i}t))$.
\end{proof}
\begin{lem}
\label{lem:2-isolated}Let $E\subset\{0,\pm1\}^{k}$ which is 2-isolated,
i.e.\ for any $v\ne w\in E$ we have that $v$ and $w$ differ in
at least two coordinates. Then
\[
\mathbb{P}(\xi\in E)\le\frac{1}{k}
\]
where $\xi=(\xi_{i})_{i=1}^{k}$ is our usual random vector.
\end{lem}

\begin{proof}
For every $v\in E$ let $\Omega_{v}\subset\{0,\pm1\}^{k}$ be the
set of all $w$ which differ from $v$ by at most one coordinate.
Then
\[
\mathbb{P}(\xi\in\Omega_{v})\ge k\mathbb{P}(\xi=v)
\]
and since the $\Omega_{v}$ for different $v$ are disjoint,
\[
\mathbb{P}(\xi\in E)=\sum_{v\in E}\mathbb{P}(\xi=v)\le\frac{1}{k}\sum_{v\in E}\mathbb{P}(\xi\in\Omega_{v})=\frac{1}{k}\mathbb{P}(\xi\in\bigcup\Omega_{v})\le\frac{1}{k}.\qedhere
\]
\end{proof}
\begin{lem}
\label{lem:PNT}The sum of $\frac{1}{p}$ over all primes $p$ between
$1$ and $n$ is less than $C\log\log n$.
\end{lem}

\begin{proof}
This is a simple corollary from the prime number theorem, which states
that there are $(1+o(1))n/\log n$ primes up to $n$. Hence for all
$k$ 
\[
\sum_{k}^{2k}\frac{1}{p}\le\frac{1}{k}\left(\frac{2k}{\log2k}-\frac{k}{\log k}+o\left(\frac{k}{\log k}\right)\right)\le\frac{C}{\log k}.
\]
Summing over $k=2^{l}$ for $l$ from $1$ to $\log n$ gives the
lemma. (in fact, a similar argument shows that this sum is $(1+o(1))\log\log n$,
but we will have no use for this extra precision).
\end{proof}
\begin{proof}
[Proof of the Theorem]The starting point is the result that 
\[
\mathbb{P}(\det M=0)\le2^{-\delta n}
\]
for some constant $\delta>0$ (this result is essentially due to of
Kahn, Komlós and Szemerédi \cite{KKS95}, though formally they only
proved the case that the coefficients are $\pm1$. For a proof for
our $\xi_{i}$ see \cite{BVW10}, which bases on \cite{TV07}. We
remark that \cite{BVW10} calculates the correct value of $\delta$,
but we have no use for this fact). Expanding the determinant by the
first row we get 
\[
\mathbb{P}\Big(\sum_{i=1}^{n}\xi_{i}d_{i}=0\Big)\le2^{-\delta n}
\]
where $d_{i}$ is the determinant of the $(i,1)^{\textrm{st}}$ minor,
i.e.\ the matrix $M$ with its $i^{th}$ column and first row removed
(we suppressed the terms $(-1)^{i+1}$ in the expansion, which we
may because the $\xi_{i}$ are symmetric to taking minus). By lemma
\ref{lem:Fourier}, for any fixed numbers $d_{i}$ and for any $x$,
$\mathbb{P}(\sum\xi_{i}d_{i}=x)\le\mathbb{P}(\sum\xi_{i}d_{i}=0)$.
Integrating over the $d_{i}$ gives
\begin{equation}
\mathbb{P}(\det M=x)\le2^{-\delta n}\qquad\forall x.\label{eq:that took a month???}
\end{equation}
Unfortunately, one cannot simply sum (\ref{eq:that took a month???})
over all squares $x$ in the possible range of values of $\det M$,
there are too many of those. So we have to take a more roundabout
way.

Let $k$ be some parameter to be fixed later. From (\ref{eq:that took a month???})
we get
\[
\mathbb{P}\bigg(\sum_{i=1}^{n-k}\xi_{i}d_{i}=0\bigg)\le2^{k-\delta n}
\]
since if the partial sum is $0$, then there is a probability of $2^{k}$
that all remaining $\xi_{i}$ are zero ($d_{i}$ are still the $(i,1)$
minors of $M$). Let $\mathcal{A}$ be the event that 
\[
\mathbb{P}\bigg(\sum_{i=1}^{n-k}\xi_{i}d_{i}=0\bigg)>2^{-n\delta/2}
\]
where here the $\xi_{i}$ are independent of $M$, so that $\mathcal{A}$
depends only on rows $2,\dotsc,n$ of $M$. By Markov's inequality,
$\mathbb{P}(\mathcal{A})\le2^{k-n\delta/2}$.

Let now $\eta\in\{0,\pm1\}^{n}$ and let $\eta'$ be identical to
$\eta$ except at one entry, say the $j^{\textrm{th}}$. Assume that
\begin{equation}
\sum_{i=1}^{n}\eta_{i}d_{i}=A^{2}\text{ and }\sum_{i=1}^{n}\eta_{i}'d_{i}=B^{2}\label{eq:both squares}
\end{equation}
for integer $A$ and $B$. Then $A^{2}-B^{2}=(A-B)(A+B)$ must be
one of $\{\pm d_{j},\pm2d_{j}\}$. Therefore every divisor of $2d_{j}$
corresponds to at most two solutions for $A$ and $B$ (up to the
signs of $A$ and $B$). Let therefore $\mathcal{B}$ be the event
that for all $j\in\{n-k+1,\dotsc,n\}$ the number of divisors of $2d_{j}$
is at most $e^{\sqrt{n}}$. By lemma \ref{lem:syrup} below and Markov's
inequality, 
\[
\mathbb{P}(\mathcal{B})\le Ck\frac{(\log n)^{2}}{\sqrt{n}}.
\]
Like $\mathcal{A}$, $\mathcal{B}$ is an event that depends only
on rows $2,\dotsc,n$ of $M$. Repeat this argument with $\eta$ and
$\eta'$ which defer in two entries, using the ``further'' clause
of lemma \ref{lem:syrup}. Let $\mathcal{C}$ be the corresponding
bad event, i.e.\ the event that for some $j\ne j'$ and some $\tau_{1},\tau_{2}\in\{\pm1,\pm2\}$
we have that $\tau_{1}d_{j}+\tau_{2}d_{j'}$ has more than $e^{\sqrt{n}}$
divisors. Then the conclusion of lemma \ref{lem:syrup} is that $\mathbb{P}(\mathcal{C})\le Ck^{2}(\log n)^{2}/\sqrt{n}$.

Denote the entries of $M$ by $m_{i,j}$. Let $\mathcal{G}$ be the
event that an $\eta$ and an $\eta'$ exist such that 
\begin{enumerate}
\item \label{enu:eta is m}$\eta_{i}=\eta'_{i}=m_{i,1}$ for all $1\le i\le n-k$.
\item $\eta$ defers from $\eta'$ by either one or two entries
\item For some integer $A$ and $B$, (\ref{eq:both squares}) holds.
\end{enumerate}
Then $\mathcal{G}$ is an event which depends only on $m_{1,1},\dotsc,m_{1,n-k}$
and $m_{i,j}$ for $i\ge2$. We now claim that 
\begin{equation}
\mathbb{P}(\mathcal{G})\le2k^{2}3^{k}2^{-\delta n/2}e^{\sqrt{n}}+\mathbb{P}(\mathcal{A}\cup\mathcal{B}\cup\mathcal{C}).\label{eq:i hate conditioned probability}
\end{equation}
This is just a summary of the previous discussion, but let us do it
in details: if $\mathcal{B}$ did not occur then for all $j\in\{n-k+1,\dotsc,n\}$,
$2d_{j}$ has no more than $e^{\sqrt{n}}$ divisors. In this case
there are only $2e^{\sqrt{n}}$ candidates for $(A,B)$ that satisfy
(\ref{eq:both squares}) for $\eta$ and $\eta'$ different at any
fixed $j$, and summing over the possibilities for $j$ gives a total
of $2ke^{\sqrt{n}}$. For each such candidate $(A,B)$ 
\[
\mathbb{P}\Big(\exists\eta\text{ s.t. }\sum_{i=1}^{n}\eta_{i}d_{i}=A^{2}\Big)\le3^{k}\max_{x\in\mathbb{Z}}\mathbb{P}\Big(\sum_{i=1}^{n-k}m_{i,1}d_{i}=x\Big)
\]
where in the left-hand side $\eta$ is assumed to satisfy assumption
(\ref{enu:eta is m}); and where the factor $3^{k}$ is simply the
number of possibilities for $\{m_{i,1}\}_{i=n-k+1}^{n}$. Applying
Lemma \ref{lem:Fourier} we get that the right-hand side is smaller
or equal to $3^{k}\mathbb{P}(\sum m_{i,1}d_{i}=0)$, and if $\mathcal{A}$
did not occur then this last probability is smaller than $2^{-\delta n/2}$.
We get that for any fixed $A$,
\[
\mathbb{P}\Big(\exists\eta\text{ s.t. }\sum_{i=1}^{n}\eta_{i}d_{i}=A^{2}\Big)\le3^{k}2^{-\delta n/2}.
\]
Summing over the $2ke^{\sqrt{n}}$ candidates finishes the case where
$\eta$ differs from $\eta'$ in just one entry. The case of two entries
is covered similarly by the event $\mathcal{C}$ (with $k$ replaced
by $\binom{k}{2}$ because we sum over two differing coordinates).
This shows (\ref{eq:i hate conditioned probability}).

The last remaining point is that if $\mathcal{G}$ did not occur,
then the probability that $\det M$ is a square is no more than $\frac{1}{k}$,
because the set of values of $m_{1,n-k+1},\dotsc,m_{1,n}$ that gives
a square is a subset of $\{0,\pm1\}^{k}$ which is $2$-isolated,
and we can apply lemma \ref{lem:2-isolated}. We conclude:
\begin{multline*}
\mathbb{P}(\det M\text{ is a square})\\
\le\frac{1}{k}+2k^{2}3^{k}2^{-\delta n/2}e^{\sqrt{n}}+2^{k-\delta n/2}+Ck\frac{(\log n)^{2}}{\sqrt{n}}+Ck^{2}\frac{(\log n)^{2}}{\sqrt{n}}
\end{multline*}
(terms 2-5 being the bound (\ref{eq:i hate conditioned probability})
on $\mathbb{P}(\mathcal{G})$ with $\mathbb{P}(\mathcal{A})$, $\mathbb{P}(\mathcal{B})$
and $\mathbb{P}(\mathcal{C})$ replaced by their bounds, in this order).
Finally we choose $k$, and taking $k=\left\lfloor n^{1/6}\right\rfloor $
gives that the right-hand side is $n^{-1/6+o(1)}$, proving the theorem.
\end{proof}
\begin{lem}
\label{lem:syrup}Let $M$ be as in the theorem. Let $X$ be the number
of divisors of $\det M$. Then
\begin{equation}
\mathbb{E}\log X\le C(\log n)^{2}.\label{eq:syrup}
\end{equation}
Further, if $d_{1}$ and $d_{2}$ are the determinants of two different
$n-1\times n-1$ first row minors of $M$, if $\tau_{1}$ and $\tau_{2}$
are in $\{\pm1,\pm2\}$ and if $Y$ is the number of divisors of $d_{1}\tau_{1}+d_{2}\tau_{2}$,
then again we have $\mathbb{E}\log Y\le C(\log n)^{2}$.
\end{lem}

\begin{proof}
The first clause is a simple corollary of Maples \cite{M}. Indeed,
theorem 1.1 of \cite{M} gives, for every prime $p$,
\begin{equation}
\mathbb{P}(p|\det M)=1-\prod_{k=1}^{\infty}(1-p^{-k})+O(e^{-\epsilon n})\label{eq:Maples}
\end{equation}
where $\epsilon>0$ is some absolute constant (the implied constant
in $O$ is also absolute, i.e.\ does not depend on $p$ or $n$).
Denote by $k(p)$ the number of times $p$ divides $\det M$ (plus
1) i.e.\ the $k$ such that $p^{k-1}|\det M$ but $p^{k}\not|\det M$.
Then
\[
\log X=\sum_{p}\log k(p)
\]
Since $|\det M|\le n!$ we have $k(p)\le C\log(n!)$, so
\[
\log X\le C(\log\log n!)|\{p\text{ prime}:p|\det M\}|.
\]
For $p<e^{\epsilon n/2}$ we use (\ref{eq:Maples}) and get 
\begin{multline*}
\mathbb{E}|\{p\text{ prime}:p\le e^{\epsilon n/2},p|\det M\}|\\
\le\sum_{p\le e^{\epsilon n/2}}\frac{C}{p}+O(e^{-\epsilon n})\le C\log\log e^{\epsilon n/2}+O(e^{-\epsilon n/2})
\end{multline*}
where the last inequality follows from lemma \ref{lem:PNT}. For $p\ge e^{\epsilon n/2}$
we note that the number of such $p$ that divide $\det M$ is no more
than 
\[
\frac{\log n!}{\log e^{\epsilon n/2}}\le C\log n.
\]
($C$ here depends on $\epsilon$, but $\epsilon$ is an absolute
constant anyway). All in all we get 
\[
\mathbb{E}|\{p\text{ prime}:p|\det M\}|\le C\log n
\]
which proves (\ref{eq:syrup}).

For the second clause (sum of determinants of two minors) we need
to examine the proof of \cite{M} a little. Assume for concreteness
that $d_{i}$ is the determinant of the $(i,1)$ minor of $M$ for
$i=1,2$. Following \cite{M}, denote by $W_{k}$ the span (over the
finite field $\mathbb{F}_{p}$) of columns $k+1,\dotsc,n$ in the
matrix $M$ without its first row (so that $W_{k}\subset\mathbb{F}_{p}^{n-1}$).
By \cite[proposition 2.1]{M} 
\begin{equation}
\mathbb{P}(\codim W_{3}\ge2)\le\frac{C}{p^{2}}+Ce^{-cn}.\label{eq:codim}
\end{equation}
(in this case, of course, $p$ would divide both $d_{1}$ and $d_{2}$
and hence also $d_{1}\tau_{1}+d_{2}\tau_{2}$). 

In the case where $\codim W_{3}=1$, we argue as follows. Let $w_{j}$
for $j=2,\dotsc,n$ be the determinant of the $n-2\times n-2$ minor
of $M$ one gets by removing the first and $j^{\textrm{th}}$ rows;
and the first two columns. Let $\eta>0$ be some parameter. Let $\mathcal{A}$
be the event that 
\begin{equation}
\bigg|\mathbb{P}\bigg(\sum_{j=2}^{n}(-1)^{j}\xi_{j}w_{j}=x\bigg)-\frac{1}{p}\bigg|\le e^{-\eta n}\quad\forall x\in\mathbb{F}_{p}.\label{eq:defA}
\end{equation}
Notice that $\mathcal{A}$ depends only the entries of $M$ different
from the first row and the first two columns (the $\xi_{j}$ in (\ref{eq:defA})
are assumed to be independent of $M$). By \cite[\S 4]{M}, for an
appropriate choice of $\eta$, independent of $p$ or $n$, 
\begin{equation}
\mathbb{P}\left(\mathcal{A}\,|\,\codim W_{3}=1\right)>1-e^{-cn}.\label{eq:Maples4}
\end{equation}
We are now finished, since $d_{i}=\sum_{j=2}^{n}(-1)^{j}m_{i,j}w_{j}$
for both $i=1,2$ and the $m_{i,j}$ are independent, so under the
event $\mathcal{A}$, for any value of $d_{1}$, the probability that
$d_{2}$ takes the value $-d_{1}\tau_{1}/\tau_{2}$ (in $\mathbb{F}_{p}$,
let us assume for a moment $p>2$) is $\frac{1}{p}+O(e^{-cn})$. With
(\ref{eq:Maples4}) we get
\[
\mathbb{P}\big(p|\tau_{1}d_{1}+\tau_{2}d_{2}\,\big|\,\codim W_{3}=1\big)=\frac{1}{p}+O(e^{-cn}).
\]
Throwing (\ref{eq:codim}) into the mix gives
\[
\mathbb{P}(p|\tau_{1}d_{1}+\tau_{2}d_{2})\le\frac{C}{p}+Ce^{-cn}.
\]
This last inequality holds also for $p=2$, of course, as it become
trivial for $C$ sufficiently large. The proof then continues as in
the single matrix case.
\end{proof}

\subsection*{Acknowledgements}

We thank Igor Rivin for enjoyable discussions and for the reference
to \cite{M}. LBS was supported by the Israel Science Foundation.
GK was supported by the Israel Science Foundation and by the Jesselson
Foundation.

\end{document}